\documentclass[a4paper,12pt]{article} 

\usepackage{cmap}							
\usepackage[T2A]{fontenc}					
\usepackage[utf8]{inputenc}					
\usepackage[final]{graphicx} 				
\usepackage{indentfirst} 					
\usepackage[usenames]{color}
\usepackage{caption}
\usepackage{setspace}

\usepackage{hyperref}

\usepackage{amsmath,amsfonts,amssymb,amsthm,mathtools}		
\usepackage{icomma}											%

\newtheorem{Theorem}{Theorem}
\newtheorem{Corollary}{Corollary}
\newtheorem{Proposition}{Proposition}

\usepackage{geometry} 
\def\R{\mathbb{R}}
\def\C{\mathbb{C}}

\let\wt=\widetilde

\date{}
	\title{On mutual arrangements of a plane real curve relative to an $M$-quartic with an oval-snake}	
	\author{S.\,Yu.~Orevkov, N.\,D.~Puchkova}

\begin{document}

	\maketitle

\begin{abstract}
An oval $O$ of a plane real algebraic quartic curve $S$ is called a snake coiling around
a real curve $C_k$ of degree $k$ if $O\cup\mathbb{R}C_k$ is isotopic to $O'\cup\mathbb{R}C_k$, 
where $O'$ is the boundary of a thickening of the embedded segment that transversally 
intersects $\mathbb{R}C_k$ at $2k$ points. In this article we prove that in this case $\mathbb{R}C_k\cup\mathbb{R}S$ 
is isotopic to $\mathbb{R}C_k\cup\mathbb{R}Q$, where $Q$ is a perturbation of 
the doubled conic. We prove analogs of this statement for real pseudoholomorphic curves 
under some additional assumptions.
\end{abstract}
	
\section* {Introduction} 
The problem of topological classification of mutual arrangements of two real algebraic curves 
(a curve of degree $m$ and a curve of degree $k$) in the real projective plane
$\mathbb{R}P^2$ belongs to the topic of the first part of Hilbert's 16th problem.
Under the assumption that these curves are in general position, the problem is solved
in the cases with $m+k \leqslant 6$, and much has been done in the cases with $m+k=7$. 
In recent papers \cite{myart1}---\cite{myart3} the second author studied mutual arrangements 
of two $M$-curves of degree 4 ($M$-quartics) intersecting in 16 distinct points located on an oval
of one curve and an oval of the other curve. This problem was posed by G.\,M.~Polotovskiy.
The arrangements studied in these papers have an oval \lq\lq{}coiling\rq\rq{} around an oval of
another curve (Fig.~\ref{Fig.Snake}; see definition below%
\footnote{In the present paper, an oval may coil around several ovals of the other curve.}).
Among the arrangements with such an oval-snake, three series determined by some additional conditions were studied
in \cite{myart1}---\cite{myart3} using rather long case-by-case considerations: first,
all topological models satisfying known restrictions
on the topology of real algebraic curves were listed, and then each model was tried to be constructed or excluded.
The proofs of non-realizability were carried out by a method based on the theory of braids and links 
proposed by the first author in \cite{LINK}.
Realizability was proved by perturbing the square of the conic in the arrangements of the conic and quartic constructed
in \cite{Polot6_1} (see Fig.~\ref{Quartic_1} below). 
For the case considered in \cite{myart2}, a complete classification was obtained.
Note that all results from \cite{myart1}--\cite{myart3} automatically extend to the case of pseudoholomorphic curves.

\begin{figure}[h!]
	\begin{minipage}[ht!]{1\linewidth}
		\centering
		\includegraphics[width=0.25\textwidth]{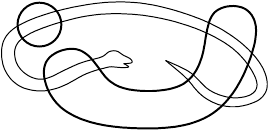}
		\caption{an oval-snake of a quartic, coiling around two ovals of another quartic.}
		\label{Fig.Snake}
	\end{minipage}
\end{figure}

However, the first author noticed that, up to isotopy in $\mathbb{R}P^2$, any arrangement of a quartic 
with a snake relative to any other curve can be obtained from an arrangement of this curve and a 
conic by a perturbation of the square of the conic, see Theorem~\ref{thm1} below.
We also prove (Theorems \ref{thm2}, \ref{thm3}) that, under some additional assumptions, 
an analogue of Theorem~\ref{thm1} holds for pseudoholomorphic curves, and that for the case 
of an $M$-quartic with a snake coiling around another quartic, the algebraic and pseudoholomorphic classifications 
coincide (Theorem \ref{thm4}).
In particular, this gives a simple proof (without tedious enumeration of logical possibilities and without 
computer calculations) of all the results of \cite{myart1}--\cite{myart3} for both 
algebraic and pseudoholomorphic curves.

In our opinion, the problem about arrangements of a curve relative to a quartic with an oval-snake
is interesting because it is one of the rare cases when topological restrictions on real algebraic 
curves are proved much easier than their analogues for real pseudoholomorphic curves.

\medskip\noindent
\textbf {Remark.}
The proof of Theorem \ref{thm1} can be considered as a simplest version of Hilbert-Rohn-Gudkov method.
In \cite{OS} this method was used to exclude some mutual arrangements of algebraic curves that 
are realizable by pseudoholomorphic curves. 
Therefore it may happen (although the chances apparently are not very high) that some
arrangements excluded by 
Theorem~\ref{thm1} but not by Theorems \ref{thm2}--\ref{thm4} (see Section 2) might by 
pseudoholomorphically realizable.

\medskip
We are grateful to G.\,M.~Polotovskiy and V.\,M.~Kharlamov for valuable discussions and advice on 
improving the article. We also thank the referee for useful remarks.

	
\section{Algebraic case}

\noindent\textbf{Definition.}
Let $C_k$ be a nonsingular plane real algebraic curve of degree $k$, and $\mathbb{R}C_k$ be the set of its real points
(connected components of are $\mathbb{R}C_k$ called {\it branches} of $C_k$).
An oval $O$ of a plane real algebraic quartic $S$ is called a \textit{snake coiling around} $C_k$ (and, more specifically,
a snake coiling around the branches of $C_k$ that $O$ intersects; denote the union of them by $B$)
if $O$ bounds a disk divided by $B$ into $2k-1$ curvilinear quadrangles and two digons (see Fig.~\ref{Fig.Snake}).
The digons are called the \textit{ends} of the snake.

\medskip

An oval-snake $O$ of a quartic $S$ coiling around $C_k$ can be equivalently defined by the condition that
$O\cup B$ is isotopic to $O'\cup B$, where $B$ is the union of branches of $C_k$ that intersect $O$, and
$O'$ is the boundary of a small thickening of a smoothly embedded segment which transversally crosses $\R C_k$ at
$2k$ points (cf.~\cite{myart1}--\cite{myart3}).

Everywhere below the expression \lq\lq{}in the disk bounded by an oval\rq\rq{} is abbreviated to \lq\lq{}inside the oval\rq\rq{}.
We say that branch of a curve is {\it free} if it is disjoint from the other curve.

\begin{Theorem}\label{thm1}
	Let $S_4$ be an $M$-quartic with an oval $O$ coiling around a curve $C_k$ of degree $k$. Then:
	
	{\rm 1)} there exists a conic $C_2$ intersecting each oval of $S_4$ and intersecting $\mathbb{R}C_k$ at $2k$
	pairwise distinct points which are inside $O$;
	
	{\rm 2)} the curve $\mathbb{R}S_4 \cup \mathbb{R}C_k$ is rigidly isotopic%
	\footnote {A rigid isotopy of algebraic curves of a given class is an isotopy through curves of this class.}
	to $\mathbb{R}\widetilde C_2^2 \cup \mathbb{R}C_k$, where $\widetilde C_2^2$ is a small perturbation of the square of $C_2$.
\end{Theorem}

\begin{figure}[h!]
	\begin{minipage}[ht!]{1\linewidth}
		\centering
		\includegraphics[width=0.33\textwidth]{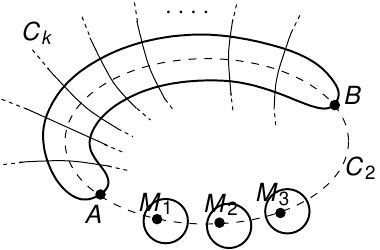}
		\caption{}
		\label{Konika}
	\end{minipage}
\end{figure}

\begin{proof}
Recall that an $M$-quartic in $\mathbb{R}P^2$ consists of four ovals lying outside each other, i.e.,
in the case under consideration $S_4$ consists of the oval-snake $O$ and three more free ovals $O_1$, $O_2$, $O_3$.
Let us choose points $A$ and $B$ on the oval-snake $O$ on the boundaries of its different ends and points $M_1$, $M_2$, $M_3$
inside $O_1$, $O_2$, $O_3$ respectively (see Fig.~\ref{Konika}). Let us consider a conic $C_2$ passing through these five points.
If the points $M_i$ are chosen in general position, then it is irreducible.
By Bezout's theorem, $C_2$ intersects each oval of $S_4$ at two points, hence $C_2$ intersects $O$ at $A$ and $B$ only.
Therefore the curve $\mathbb{R}C_k$ intersects $\mathbb{R}C_2$ at $2k$ points inside the oval-snake $O$.
The first assertion is proved.
	
Let us consider the pencil of quartics
\begin{equation}
		S_4(t)= S_4+t{C_2}^2=0.
		\label{eq1}
\end{equation}
Then $S_4(0)=S_4$ and points of
intersections of $\R S_4(t)$ with $\R C_k$ cannot appear or disappear
during a continuous variation of the parameter $t$
	
Let us choose the sign of $t$ so that the interiors of ovals of the quartic shrink when $|t|$ increases.
When $|t| \rightarrow \infty$, it is clear from (\ref{eq1}) that the quartic $S_4(t)$ approaches the square ${C_2}^2$ 
of the conic $C_2$, hence $\mathbb{R}S_4$ is isotopic to a small perturbation of $C_2^2$.
The rigidity of the isotopy follows directly from its construction.
\end{proof}

\begin{Corollary}\label{cor11}
	Let $S_4$ be an $M$-quartic with an oval $O$ coiling around a curve $C_k$ of degree $k$. Then all the free ovals of $S_4$ lie
	in the same connected component of the complement of $O \cup \mathbb{R}C_k$,
	and the boundary of this component includes both arcs of $O$ which bound the ends of $O$. 
\end{Corollary}

Theorem \ref{thm1} easily yields a classification of mutual arrangements of a curve $C_k$ and an $M$-quartic
with a snake coiling around $C_k$ as soon as a classification of arrangements of a curve of degree $k$
intersecting a conic at $2k$ real points is known.

As an example, let us consider the case when $k = 3$ and $C_3$ is an $M$-cubic.
Classification \cite{Polot6_2} of mutual arrangements of $M$-cubics and conics with six common points is very simple;
see Fig.~\ref{CubicConic}.

\begin{figure}[ht!]
	\begin{minipage}[ht!]{1\linewidth}
		\centering
		\includegraphics[width=0.98\textwidth]{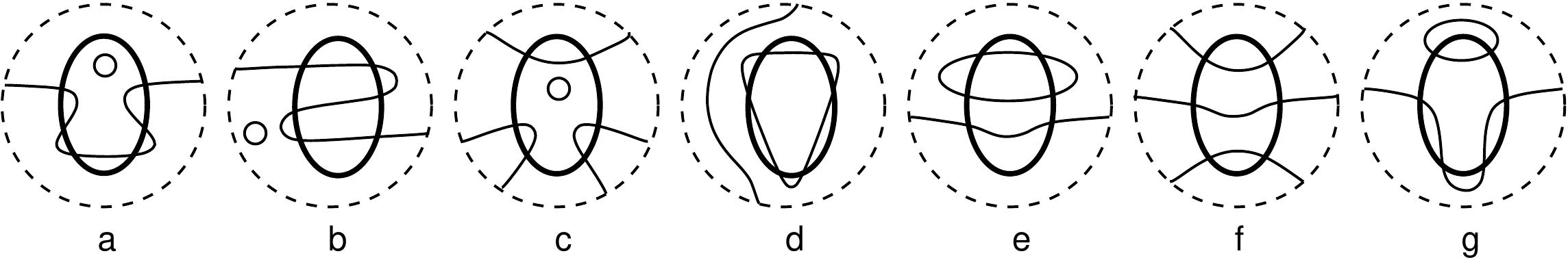}
		\caption{}
		\label{CubicConic}
	\end{minipage}
\end{figure}

\begin{figure}[ht!]
	\begin{minipage}[ht!]{1\linewidth}
		\centering
		\includegraphics[width=0.7\textwidth]{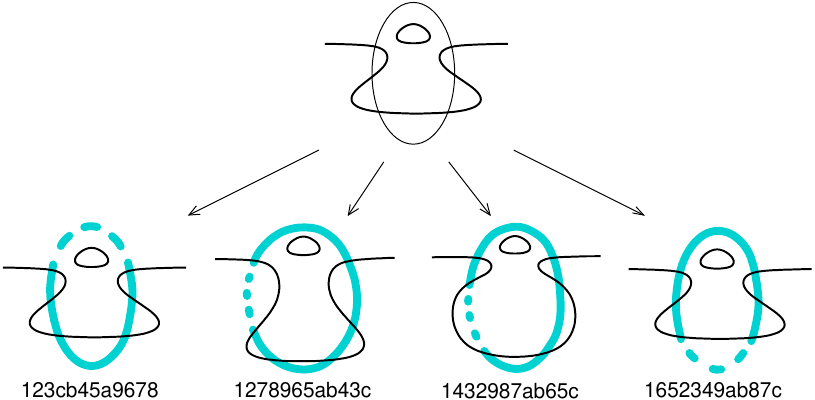}
		\caption{}
		\label{buildFrom3To6}
	\end{minipage}
\end{figure}

Constructions coming from the arrangement of Fig.~\ref{CubicConic}a are shown in Fig. \ref{buildFrom3To6}.
Namely, the cubic cuts the conic into six arcs. 
We choose four points on one of these arcs and consider the pencil of conics passing through them.
The union of the initial conic with a near conic from this pencil form a closed chain of four digons.
We may perturb the union of the conics so that each of the digons is transformed into an oval. Then
we obtain the arrangements of quartic and cubic shown in Fig.~\ref{buildFrom3To6}.

Note that the list in \cite[\S5]{OrevkovConstruction} includes all
the arrangements obtained from Fig.~\ref{buildFrom3To6} (their codes according to \cite{OrevkovConstruction}
are given in Fig.~\ref{buildFrom3To6}). 
Similarly, all other arrangements from \cite[\S5]{OrevkovConstruction} with an oval-snake (namely,
123cb478965a, 1278963cb45a, 1432985cb67a, 123c/9678/b45a, 1278/b43c/965a) are obtained
from Fig.~\ref{CubicConic}b, \ref{CubicConic}c, whereas the second and third arrangements in Fig.~1 in \cite{OrevPolot}
are obtained from Fig.~\ref{CubicConic}d.
Evidently, there are no other arrangements with a snake in \cite{OrevkovConstruction}, \cite{OrevPolot}.

In the same way, the classification of arrangements of conic and quartic in \cite{Polot6_1} (Fig.~\ref{Quartic_1}) and in
hard-to-find texts \cite{Polot6_2}, \cite{Polot6_3} (see Fig.~\ref{Quartic_2}, \ref{Quartic_3}) yields a classification of
arrangements of an $M$-quartic with a snake coiling around another quartic. The classification of arrangements of conic and 
$M$-quintic \cite{Orev_52_2}, \cite{Orev_52_1} yields a classification of $M$-quartics with a snake coiling around a branch 
of an $M$-quintic.  In particular, one obtains 84 (97, 20, 2) pairwise distinct isotopy types of arrangements with a snake 
coiling around one (respectively, two, three, four) ovals of another $M$-quartic.
Note that sometimes different initial arrangements of $C_k$ and $C_2$ produce isotopic arrangements of $C_k$ and $S_4$.

\begin{figure}[h!]
	\begin{minipage}[h!]{0.295\linewidth}
		\center{\includegraphics[width=1\textwidth]{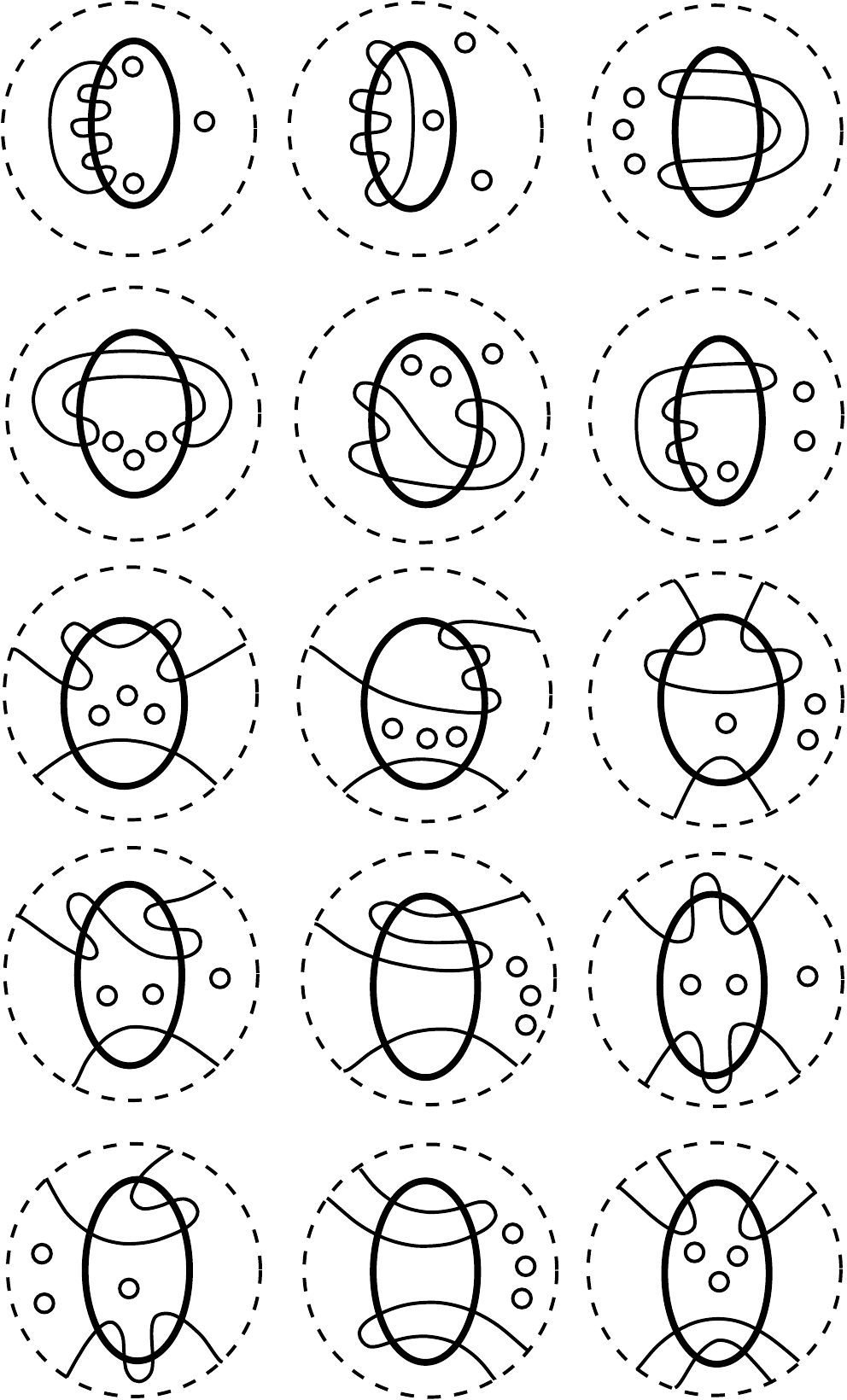} \\ } 
	\caption{}
	\label{Quartic_1}
\end{minipage}
\hfill
\begin{minipage}[ht!]{0.295\linewidth}
	\center{\includegraphics[width=1\linewidth]{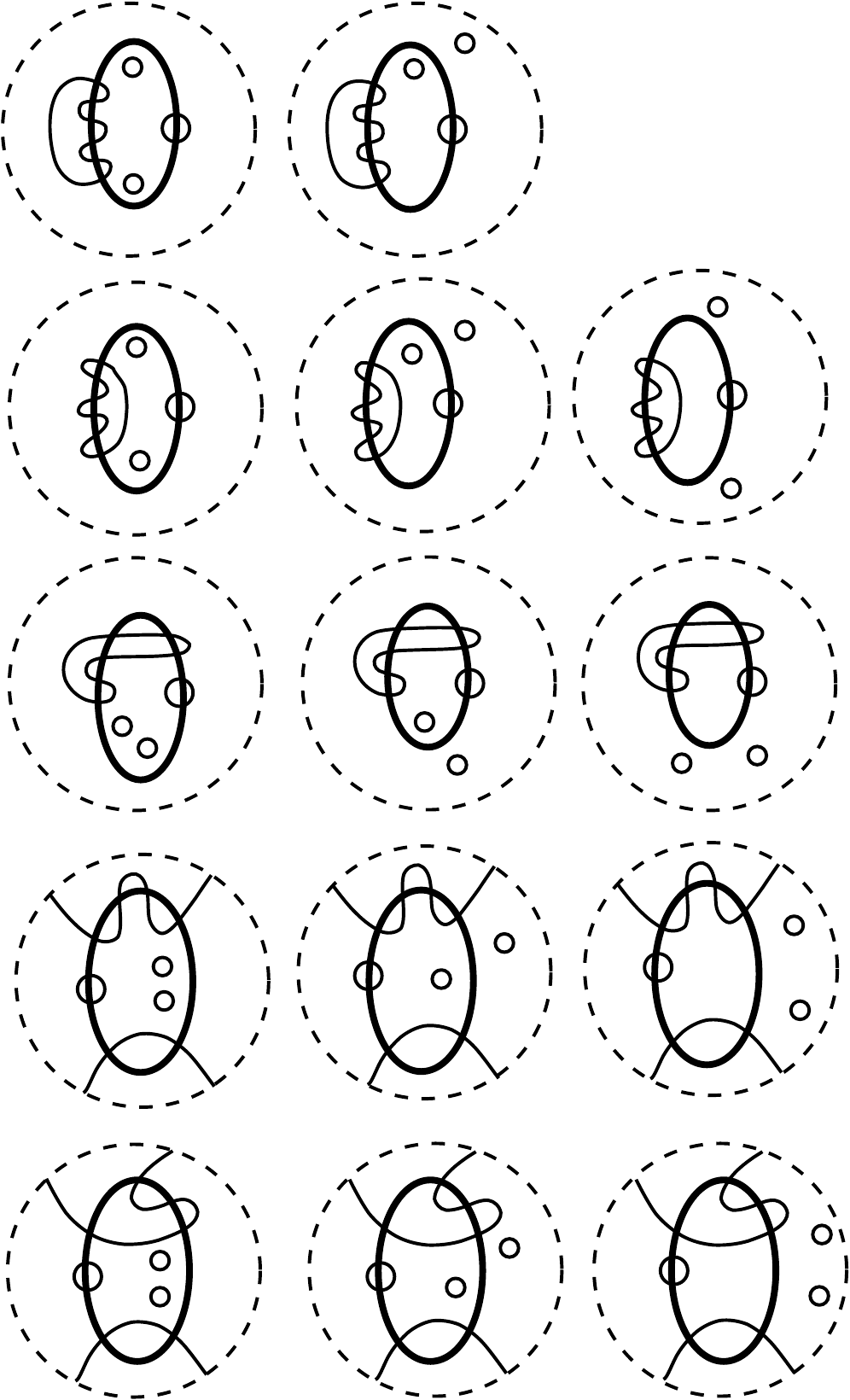} \\ } 
\caption{}
\label{Quartic_2}
\end{minipage}
\hfill
\begin{minipage}[ht!]{0.295\linewidth}
\center{\includegraphics[width=1\linewidth]{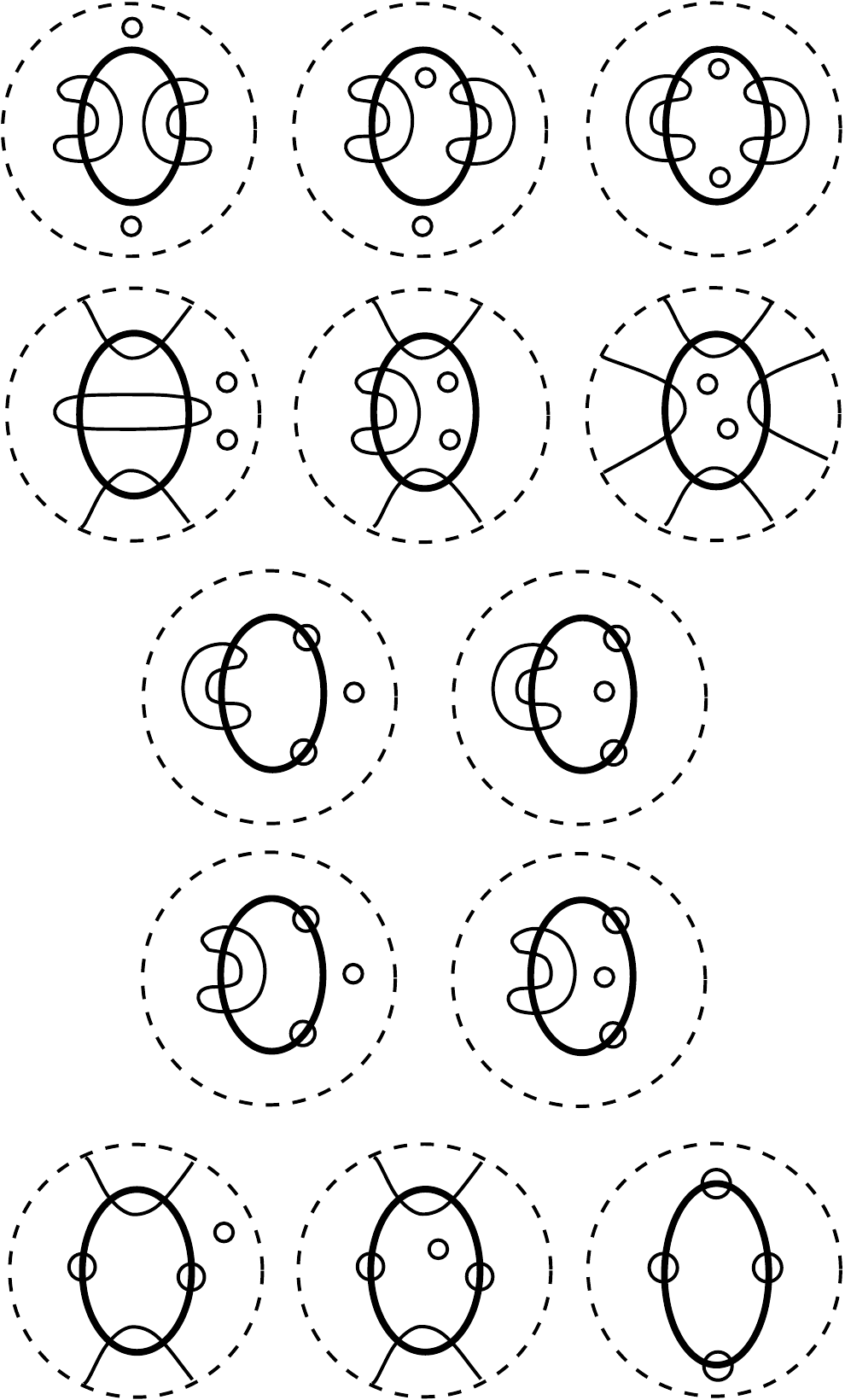} \\} 
\caption{}
\label{Quartic_3}
\end{minipage}
\end{figure}


\section {Pseudoholomorphic case}

\subsection{Application of an auxiliary conic}

In the case of pseudoholomorphic curves (for real pseudoholomorphic curves see, for example, \cite{GAFA}) 
the definition of an oval-snake is exactly the same as in \S1 above.

\begin{Theorem}\label{thm2}
Let $S_4$ be a real pseudoholomorphic $M$-quartic with oval $O$ coiling around a real pseudoholomorphic 
curve $C_k$ of degree $k$. Then:
	
{\rm 1)} there exists a real pseudoholomorphic {\rm(}with respect to the same almost complex structure{\rm)} 
conic $C_2$ intersecting each oval of $S_4$ and intersecting curve $\mathbb{R}C_k$ at $2k$ pairwise distinct points inside $O$;
	
{\rm 2)} there are no free ovals of $C_k$ inside the ends of the oval-snake and inside free ovals of $S_4$;
	
{\rm 3)} there exists an isotopy $\{O_t\}_{t\in[0,1]}$ of the oval-snake $O=O_0$ such that $O_1$ is an oval 
of a small perturbation of the doubled $C_2$, 
the intersection of $O_t$ with the non-free branches of $C_k$ is transverse, and $O_t$ lies inside $O$ for all $t>0$;
	
{\rm 4)} if there are no ovals of $C_k$ inside $O$, then $\mathbb{R}S_4 \cup \mathbb{R}C_k$ is isotopic to 
$\mathbb{R}Q\cup \mathbb{R}C_k$, where $Q$ is a small perturbation of the doubled $C_2$.

\end{Theorem}

\begin{proof}
1) Let us choose five points as in the proof of Theorem \ref{thm1}.
By virtue of Gromov's results \cite{Gromov}, an irreducible pseudoholomorphic conic $C_2$ passes through the chosen points.
As in the proof of Theorem \ref{thm1}, this conic intersects $C_k$ at $2k$ points inside the oval-snake $O$.
	
2) If such a free oval of $C_k$ existed, then, choosing one of the five points inside it, we would obtain an 
arrangement of $C_2$ and $S_4$ that contradicts Bezout's theorem.
	
The last two statements easily follow from the first two.
\end{proof}

\begin{Corollary}\label{cor2}
	For pseudoholomorphic curves, an analogue of Corollary 1 holds.
\end{Corollary}


\subsection{Application of a maximum pencil of lines}

Let us prove the assertion \,4) of Theorem \ref{thm2} under other assumptions.

Let $F_k$ be a real pseudoholomorphic curve of degree $k$, $M\in \mathbb{R}P^2$ be a point not lying on $\mathbb{R}F_k$, 
and $L_M$ be a pencil of lines centered at $M$.
We call the pencil $L_M$ \textit {maximal for $\mathbb{R}F_k$} if any line in this pencil intersects $\mathbb{R}F_k$ 
at least in $k-2$ points. We call an interval of $L_M$ \textit {maximal for $\mathbb{R}F_k$} 
if the same conditions are satisfied by the set of lines of this interval.

\medskip\noindent
\textbf{Definition.} We say that ovals $\alpha$ and $\beta$ of a curve $F_k$ are \textit{neighboring for a pencil} $L_M$
if there exists an interval $I$ of $L_M$ bounded by lines $t_\alpha, t_\beta \in L_M$ tangent to $\alpha$ and $\beta$
at points $T_\alpha$ and $T_\beta$ respectively, such that no line from $I$ is tangent to $F_k$ and,
in sufficiently small neighborhoods of $T_\alpha$ and $T_\beta$, the ovals
$\alpha$ and $\beta$ do not intersect lines from $I$.
Further, we say that the pencil $L_M$ \textit {sweeps the ovals of $F_k$} if all the ovals of $F_k$ form
a cyclic sequence such that consecutive ovals are neighboring for $L_M$.

\begin{Theorem}\label{thm3}
Let $S_4$ be a real pseudoholomorphic $M$-quartic with an oval-snake $O$ coiling around a real pseudoholomorphic 
curve $C_k$ of degree $k$. If there exists a pencil of lines sweeping the ovals of $S_4$ such that some neighborhoods 
of the closures of intervals
of this pencil between lines tangent to neiboring ovals are maximal for $\mathbb{R}S_4\cup \mathbb{R}C_k$,
then $\mathbb{R}S_4\cup \mathbb{R}C_k$ is isotopic to $\mathbb{R}Q\cup \mathbb{R}C_k$, where $Q$ is a small
perturbation of the double of a conic $C_2$ intersecting $C_k$ at $2k$ pairwise distinct points inside $O$.
\end{Theorem}

\begin{proof}
Suppose that a pencil of lines $L_M$ satisfies the formulated properties.
Then (see Proposition 2.2 \cite{GAFA}) among the $(k+4)$-strand braids that can be obtained by perturbing the singularities
of the intersection of $\mathbb{C}S_4 \cup \mathbb{C}C_k$ and the complexification of $L_M$, there is a quasi-positive
braid $b$ (for a description of the construction of the braids, see, e.g., \cite{LINK} or \cite{GAFA}).
	
As in the proof of Theorem \ref{thm2}, we construct an irreducible pseudoholomorphic conic $C_2$ which intersects
each oval of $S_4$ in two points and intersects $C_k$ in $2k$ points inside $O$.
	
Let $\alpha$ and $\beta$ be neighboring ovals of $S_4$ (in the order they are swept by $L_M$),
and let $D$ be an arc of $\R C_2$ with endpoints $D_\alpha \in \alpha$ and $D_\beta\in \beta$ which has no other 
common points with $\R S_4$. 
Let $l_\alpha$ and $l_\beta$ be the lines from the pencil $L_M$ that intersect $\alpha$ and $\beta$ so that $D$
is contained in the region bounded by these lines. 
Let us apply modification \lq\lq $\supset \ \subset \ \longrightarrow \times $\rq\rq\ in this region (see Fig. 7): 
delete the parts of $\alpha$ and $\beta$ that fall into this region and then crosswise connect the ends of the remaining 
parts of $\alpha$ and $\beta$ by two arcs that transversally cross each other at one point. 
We choose the new arcs so that they are disjoint from $C_k$.
This is possible since $D$ is disjoint from $C_k$ (all the $2k$ common points $C_2$ and $C_k$ lie inside $O$).
	
	\begin{figure}[ht!]
		\centering
		\begin{minipage}[h]{0.45\linewidth}
			\center{\includegraphics[width=0.9\linewidth]{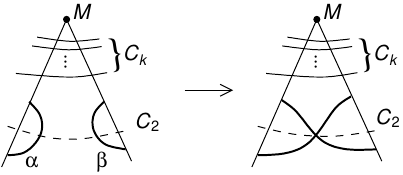} \\ }
			\caption{the modification \lq\lq $\supset \ \subset \ \longrightarrow \times $\rq\rq.}
			\label{Snake_PointD}
		\end{minipage}
		\hskip 1cm
		\begin{minipage}[h!]{0.3\linewidth}
			\center{\includegraphics[width=0.8\linewidth]{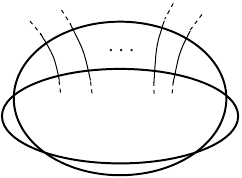} \\}
			\caption{$C_k\cup\widetilde S_4$.}
			\label{Snake_4OvalsAndXes}
		\end{minipage}
	\end{figure}
	
Since the pencil $L_M$ is maximal for $\mathbb{R}S_4\cup \mathbb{R}C_k$ in the intervals specified in
the hypothesis of the theorem, the described modification does not change the braid $b$ (see Corollary 3.2 in \cite{GAFA}).
Therefore, the resulting arrangement will still be pseudoholomorphically realizable.
Once we perform such modifications for each of the four pairs of noighboring ovals, we obtain an arrangement of $C_k$
and a quartic $\widetilde S_4$ with four double points. By the genus formula, $\widetilde S_4$ cannot be irreducible,
therefore, it splits into two conics (Fig.~\ref{Snake_4OvalsAndXes}), each of which intersects $C_k$ in $2k$ points.
	
	It remains to prove that there are no ovals of $C_k$ in the digons bounded by the curve $\widetilde S_4$.
Indeed, if there were such an oval, we would obtain a contradiction with Bezout's Theorem for a conic passing through
a point inside this oval and all the four double points of $\widetilde S_4$.
\end{proof}

\smallskip\noindent
\textbf{Remark.}
For most cases (all but three) of the arrangement of an oval-snake coiling around an oval of another $M$-quartic 
considered in \cite{myart1}--\cite{myart3}, a pencil satisfying the hypothesis of Theorem \ref{thm3} exists: 
its center can be chosen outside the oval-snake and inside the oval $\alpha$ of the second quartic coiled by it, 
so that for each end of the oval-snake, the line from the pencil intersecting this end intersects the oval $\alpha$ at four points.


\subsection{Application of Viro-Kharlamov's theorem}

For the case $k=4$, the condition that there is no oval of the curve $C_k$ inside the oval-snake (see the last statement 
of Theorem \ref{thm2}) can be proved using the results by Viro and Kharlamov \cite{KhV}, 
which generalize classical congruences modulo 8 to singular curves.
Therefore, the following theorem holds.

\begin{Theorem}\label{thm4}
Let $S_4$ be a real pseudoholomorphic $M$-quartic with an oval-snake coiling around a real pseudoholomorphic quartic $C_4$. 
Then the union $\mathbb{R}S_4 \cup \mathbb{R}C_4$ is isotopic to $\mathbb{R}Q \cup \mathbb{R}C_4$, 
where $Q$ is a small perturbation of a double conic.
\end{Theorem}

Before proving Theorem \ref{thm4}, we present the results we need from \cite{KhV} and explain why they 
apply to pseudoholomorphic curves.

Let $F$ be a real algebraic curve of an even degree $2k$ in the projective plane, possibly singular 
(for example, reducible) but without multiple components and without isolated points of $\R F$.
Let $\Gamma=\Gamma(F)$ be the union of the connected components of $\R F$ which contain singular points of $F$.
Let $\R P^2_+$ (resp. $\Gamma_+$) be one of the two closed subsets of $\R P^2$ whose common boundary is $\R F$ (resp. $\Gamma$). 
We assume that $\R P^2_+$ and $\Gamma_+$ are compatible in the sense that 
$V\cap\R P^2_+=V\cap\Gamma_+$ for some neighborhood $V$ of $\Gamma$.

Let $F_1,\dots,\,F_n$ be the irreducible components of $F$, and $\wt F_1,\dots,\,\wt F_n$ be their normalizations (non-singular models).
We say that $F$ is an $(M-r)$-curve if $\wt F_i$ is an $(M-r_i)$-curve ($i=1,\dots,\,n$) and $r=r_1+\dots+r_n$.
In particular, $F$ is an $M$-curve if all the $\wt F_i$ are $M$-curves.
We say that $F$ is a curve of Type I if all the $\wt F_i$ are curves of Type I (i.e. the sets $\wt F_i\setminus\R\wt F_i$
are disconnected).
Otherwise, we will say that $F$ is a curve of Type II.

Under certain additional conditions on a curve $F$, it is proved in
\cite[(3.A), (3.B)]{KhV} that the following analogs of the Gudkov--Rokhlin congruence, 
the Gudkov--Krakhnov--Kharlamov congruence, and the Kharlamov--Marin congruence hold:
\begin{align}
	\chi(\R P^2_+)    \equiv k^2 + q  \,\;\;\;\;\mod 8, \qquad&\text{ if $F$ is an $M$-curve,             }\label{eq.GR}\\
	\chi(\R P^2_+)    \equiv k^2 + q\pm 1 \!\!\!\mod 8, \qquad&\text{ if $F$ is an $(M-1)$-curve,         }\label{eq.GKK}\\
	\chi(\R P^2_+)\not\equiv k^2 + q +  4 \!\!\!\mod 8, \qquad&\text{ if $F$ is an $(M-2)$-curve of Type II, }\label{eq.Kh}
\end{align}
where $q$ depends only on the topology of the pair $(\R P^2,\,\Gamma_+)$ and, if the singularities are more 
complex than simple double points, on their types and locations on $\Gamma$ (see~\cite[\S\S2.3--2.4]{KhV}).

In the case of nodal curves (the only case we need here) the aforementioned additional conditions are especially simple:
\begin{itemize}
	\item[(I)] 
	all singularities of $F$ are real ordinary double points with real tangents;
	\item[(II)]
	each real branch of $F$ (a smoothly immersed circle) intersects the union of the other branches at $d\equiv 0$ mod 4 
	points if it is contractible in $\R P^2$, and at $d\equiv(-1)^{k+1}$ mod 4 points if it is not contractible 
	(recall that $\deg F=2k$).	
\end{itemize}

\medskip\noindent
\textbf{Remark.}
There are some errors in \cite[\S4.3]{KhV} in the formulation of Condition (II) and in the description of the right-hand side of
\eqref{eq.GR}--\eqref{eq.Kh} for nodal curves. They are corrected in \cite{VO}.

\medskip

O.\,Ya.~Viro \cite{Viro} noticed that many topological results on non-singular plane real algebraic curves, 
including congruences modulo 8, and their proofs automatically extend to the so-called {\it flexible curves}, 
the definition of which is given in \cite[\S1]{Viro}.
This observation also applies to the results of \cite{KhV}, if flexible curves with singularities are defined as follows.

\medskip\noindent
\textbf{Definition.} 
1. A subset $X\subset\C P^2$ is called a {\it real surface
with complex-analytic singularities}, if it is a complex-analytic curve in a neighborhood $U$ of its finite subset $\Sigma$,
it is a smooth oriented real two-dimensional submanifold of $\C P^2$ outside $\Sigma$, and its orientation in $X\cap U$ is the natural orientation of a complex curve. 
Then $X$ uniquely decomposes into a union $X=X_1\cup\ldots\cup X_n$, where each $X_i$ is the image of a connected compact 
Riemann surface $\wt X_i$ under a continuous mapping $\nu_i:\wt X_i\to\C P^2$ such that $\wt\Sigma_i=\nu_i^{-1}(X_i\cap\Sigma)$ 
is finite and the restriction $\nu_i|_{\wt X_i\setminus\wt\Sigma_i}$ is an embedding. 
We call the sets $X_i$ {\it irreducible components} of $X$. We define the genus of $X_i$ to be the genus of $\wt X_i$,
and we set its degree to be equal to $m_i$ such that $[X_i] = m_i[\C P^1]$ in $H_2(\C P^2)$.

\smallskip
2. (Cf.~\cite[\S1]{Viro}.)
Let $X$ be a real surface in $\C P^2$ with complex-analytic singularities and $X_1,\dots,\,X_n$ be its irreducible components, 
$m_i$ be the degree of $X_i$, and $g_i$ be its genus.
We call $X$ a {\it flexible curve of degree $m$} if the following conditions are satisfied:
\begin{itemize}
	\item[(i)] $m = m_1+\dots+m_n$;
	\item[(ii)] the genus formula
	$g_i+\sum_{p\in X_i\cap\Sigma}\delta_p(X_i) = \frac12(m_i-1)(m_i-2)$ holds for all $i=1,\dots,\,n$,
	where $\delta_p(X_i)$ denotes the delta-invariant of the singularity of $X_i\cap U$ at $p$;
	\item[(iii)] $X$ is invariant under complex conjugation $\operatorname{conj}:\C P^2\to\C P^2$;
	\item[(iv)] the field of tangent planes on $X\cap\R P^2$ can be deformed in the class of conj-invariant
	planes into the field of complex lines tangent to $X\cap\R P^2$, so that the deformation is identical
	in a neighborhood of $\Sigma$.
\end{itemize}

Real nodal pseudoholomorphic curves in $\C P^2$ are conj-equivariantly isotopic to flexible nodal curves, so the congruences
\eqref{eq.GR}--\eqref{eq.Kh} under Conditions (I) and (II) apply to them also.
Indeed, a conj-equivariant isotopy can make a curve complex-analytic in neighborhoods of nodes, since all the
self-intersections are positive. A continuous deformation of the almost complex structure into the standard complex
structure \cite[\S2.3]{Gromov} yields condition (iv). Condition (ii) follows from the adjunction formula for symplectic
surfaces (see Lemma 1.5.1 in \cite{IvSh}).

\medskip\noindent
{\it Proof of Theorem \ref{thm4}}.
Let $C_2$ be the conic from Theorem \ref{thm2} (see Fig.~\ref{Konika}) and $Q$ be the $M$-perturbation of its double,
which has an oval-snake. Denote $S_4\cup C_4$ and $Q\cup C_4$ by $F$ and $G$.
By the third item of Theorem \ref{thm2}, the sets $\Gamma(F)$ and $\Gamma(G)$ are isotopic, and therefore the value $q$
on the right-hand side of the congruences 
is the same for both curves. Moreover, for each of the curves $F$ and $G$, with an appropriate choice of $\R P^2_+$,
the following equality holds:
\begin{equation}\label{eq.thm4}
	\chi(\R P^2_+)=\Phi+p-n, 
\end{equation}
where $p$ (resp. $n$) is the number of free ovals of the non-snake quartic $C_4$ lying outside (resp. inside) the oval-snake,
and the quantity $\Phi$ is the same for both curves.

Therefore, in the case when $C_4$ is an $M$-quartic, \eqref{eq.GR} and \eqref{eq.thm4} imply that $p(F)=p(G)$ and $n(F)=n(G)=0$,
which means that there are no free ovals of $C_4$ inside the oval-snake, and the result follows from the last item
of Theorem \ref{thm2}.

In the case when $C_4$ is an $(M-r)$-quartic of Type II ($r=1,2$), we use the fact (see Proposition \ref{prop.C2+C4} below)
that any arrangement of the non-free branches of $C_4$ and $C_2$ is realizable by an $M$-quartic and a conic.
It follows that the left-hand side in \eqref{eq.GKK}, \eqref{eq.Kh} for the curve $G$ is equal to $k^2 + q - r$,
and the conclusion of the proof is the same as in the $M$-case.

Finally, if $C_4$ is an $(M-2)$-quartic of Type I, Rokhlin's formula for complex orientations \cite[(3.13)]{Viro}
applies to curves obtained from $F$ and $G$ by smoothing their singular points in a manner agreed with complex orientations.
When moving one free oval of $G$ inward the oval-snake, this formula fails.
\qed
\medskip

Recall that a plane real quartic $C_4$ is called {\it hyperbolic} if $\R C_4$ consists of two ovals,
one of which lies inside the other one (this is equivalent to $C_4$ being an $(M-2)$-quartic of Type I).

\begin{Proposition}\label{prop.C2+C4}
	Let $C_2$ and $C_4$ be non-singular real pseudoholomorphic {\rm(}for example, real algebraic{\rm)}
	conic and quartic in the projective plane which intersect each other at $8$ real points.
	If $C_4$ is not hyperbolic, then the arrangement of $\R C_2\cup\R C_4$ on $\R P^2$ is, up to isotopy,
	either as shown in Fig.~\ref{Quartic_1}--\ref{Quartic_3}, or it is obtained from these arrangements by
	removing some free ovals of the quartic.
	All these arrangements are algebraically realizable.
\end{Proposition}

\begin{proof}
	Bezout's theorem for auxiliary lines implies that the mutual arrangements of non-free ovals can only be as
in Fig.~\ref{Quartic_1}--\ref{Quartic_3}, and that the free ovals can appear in at most in two components of the complement
of the non-free ovals. Moreover, one of these components lies inside $\Gamma_+$ while the other lies outside $\Gamma_+$.
Therefore, in cases when Conditions (I) and (II) are satisfied (Fig.~\ref{Quartic_1} and the first two rows
in Fig.~\ref{Quartic_3}), the result follows from \eqref{eq.GR}--\eqref{eq.Kh}, where $q$ can be found from the realizable arrangements.
For the arrangement of the non-free ovals from the first row in Fig.~\ref{Quartic_2}, in the case of an $M$-quartic the result
follows from the Rokhlin's formula for complex orientations \cite[(3.13)]{Viro} applied to the smoothing of all double points
according to complex orientations (see~\cite[\S3.1]{OrevPolot}).
In other cases it suffices to apply Bezout's theorem for auxiliary lines.
	
For an algebraic realization, see \cite{Polot6_1}--\cite{Polot6_3}. It can also be easily obtained by perturbing
a quartic maximally tangent to a conic at singularities of type $A_2$ or $A_4$ (see~\cite{OrevkovConstruction}--\cite{OrevPolot}).
\end{proof}

From this proposition, by virtue of Theorems 1 and 4, it follows that in the case when an $M$-quartic with a snake coils
around a non-hyperbolic quartic, the algebraic and pseudoholomorphic isotopic classifications of the unions of such curves coincide.
In the case when $C_4$ is a hyperbolic quartic, these classifications also coincide (the algebraic classification of the
arrangements of a hyperbolic quartic and a conic intersecting each other at 8 real points is obtained in
\cite[Theorem 2.4]{Polot6_2}, and its coincidence with the pseudoholomorphic classification can be proved in the
same way as Proposition~\ref{prop.C2+C4}).

\bigskip
Steklov Math. Inst., Moscow, Russia (S.O.)

IMT, l'universit\'e de Toulouse, Toulouse, France (S.O.)

Higher School of Economics, Nizhny Novgorod, Russia (N.P.)

\end{document}